\newcommand{\labbel}{\label}

\documentclass[microtype]{jloganxx}

\title{Compactness  of $ \omega ^ \lambda $ for $\lambda$ singular}

\author{Paolo Lipparini}
\givenname{Paolo}
\surname{Lipparini}
\address{Dipartimento di Matematica\\Viale della Ricerca Scienfitica\\II Universit\`a di Roma (Tor Vergata)\\I-00133 ROME ITALY}
\email{lipparin@axp.mat.uniroma2.it}
\urladdr{http://www.mat.uniroma2.it/\textasciitilde lipparin}



\newtheorem{thm}{Theorem}
\newtheorem{cor}[thm]{Corollary}          

\newtheorem{proposition}[thm]{Proposition}
\newtheorem{theorem}[thm]{Theorem}

\theoremstyle{definition}

\newtheorem{definition}[thm]{Definition} 

\newcommand{\brfrt}{\hspace{0 pt}}

\newcommand{\m}{\mathfrak}
\newcommand{\iit}{\mathfrak{iit}}

\begin{document}

\begin{abstract}   
We characterize the compactness properties of
the 
product of $ \lambda $ copies of the space $ \omega$ with the discrete topology,
dealing in particular with the case $\lambda$ singular,
using regular and uniform ultrafilters,
infinitary languages and
nonstandard elements.
We also deal with
products of  uncountable  regular cardinals with the order topology.
\end{abstract}

\maketitle

The problem of determining the compactness properties satisfied by powers
of the countably infinite discrete topological space $ \omega$ originates
from Stone \cite{St}, who proved  that
$ \omega ^{ \omega _1} $ is not normal,
 hence, in particular,  not Lindel\"of. 
More generally, Mycielski \cite{M}
showed that $\omega^ \kappa $ is not finally $\kappa$-compact, 
for every infinite cardinal $\kappa$ strictly less than the first weakly inaccessible
cardinal. Recall that a topological space is said to be \emph{finally $\kappa$-compact} if any open cover has a subcover of
cardinality strictly less than $\kappa$. \emph{Lindel\"ofness} is the same as 
final $ \omega_1$-compactness. 
Previous work on the subject had been also done  by A. Ehrenfeucht, 
P. Erd\"os, A. Hajnal and J. \L o\'s; see  \cite{M} for details.
Mycielski's result cannot be generalized to arbitrarily large cardinals:
if $\kappa$ is weakly compact then
$\omega^ \kappa $ is indeed finally $\kappa$-compact:
see  Keisler and Tarski \cite[Theorem 4.32]{KT}. Related work is due to
D. V. \v{C}udnovski\u{\i}, W. Hanf, D. Monk, D. Scott,  
S. Todor\v cevi\'c and S. Ulam, among many others.
With regard to powers of $ \omega$ a more refined 
result has been obtained by
Mr\'owka  who, e.~g., in \cite{Mr2} showed that if the infinitary language
$\mathcal L _{ \omega _1, \omega } $
is $ ( \kappa  , \kappa   )$-compact, then
$\omega^ \kappa $ is  finally $\kappa$-compact.
This is a stronger result since Boos \cite{B} showed that it is possible that
$\mathcal L _{ \omega _1, \omega } $
is $ ( \kappa  , \kappa   )$-compact, even, 
that 
$\mathcal L _{ \kappa , \omega } $
is $ ( \kappa  , \kappa   )$-compact, 
 without $\kappa$ being weakly compact.
To the best of our knowledge the gap between 
Mycielski's and  Mr\'owka's results had not been 
exactly filled until we showed in 
\cite{cpnw} that   Mr\'owka gives the optimal bound, that is, for $ \kappa $
regular, 
$\omega^ \kappa $ is  finally $\kappa$-compact
if and only if $\mathcal L _{ \omega _1, \omega } $
is $ ( \kappa  , \kappa   )$-compact.
The aim of 
the present note is to  show that the result holds also for a singular 
 cardinal $\kappa$. 
In order to give the proof, we need 
to use  uniform and regular ultrafilters,  as well as
nonstandard elements; in particular, we shall introduce some related principles
which may have independent interest and
which, in a sense, measure ``how hard it is'' 
to exclude the uniformity of some ultrafilter, on one hand,
or to omit the existence of a nonstandard element in some
elementary extension on the other hand.
 A  large part of our methods work for arbitrary
regular cardinals in place of $ \omega$;
in particular, at a certain point, we shall make good use of a notion 
whose importance has been hinted in Chang \cite{Ch} and which we call here
being ``$ \mu $-nonstandard''; in the particular case $ \mu= \omega $ we get back 
the classical notion.
These techniques  allow us  to provide a characterization of the
compactness properties of products of (possibly uncountable) regular cardinals with the order topology.
This seems to have some interest since, as far as we know, all previously known  results of this kind
have dealt with cardinals endowed with the discrete topology
(of course, the two situations coincide in the case of $ \omega$).

The following theorem has been proved in \cite{cpnw}
in the case when $\lambda$ is regular, with a slightly simpler condition 
in place of  (3) below.
We shall prove the theorem  here for arbitrary $\lambda$. All the relevant definitions shall be given shortly after the statement.

\begin{thm} \labbel{logb}  
The space  $ \omega^ \lambda $
is finally $\lambda$-compact if and only if 
$\mathcal L _{ \omega _1, \omega } $
is $ ( \lambda , \lambda  )$-compact.
More generally, if $\kappa \geq \lambda $ then 
the following conditions are equivalent.
 \begin{enumerate} 
  \item  
$ \omega^ \kappa $ is $[ \lambda , \lambda  ]$-compact.
\item
The language $\mathcal L _{ \omega _1, \omega } $
is $ \kappa  $-$( \lambda , \lambda  )$-compact.
  \item
$(\lambda, \lambda ) {\not\Rightarrow}\!^ \kappa \omega $.
 \end{enumerate} 
\end{thm}

Unexplained notions and notation are standard; see, e.~g., 
Chang and Keisler \cite{CK}, Comfort and Negrepontis \cite{CN} and
Jech \cite{J}.  
Throughout, $\lambda$, $\mu$, $\kappa$ and $\nu$ are infinite cardinals, 
$\alpha$, $\beta$ and $\gamma$ are ordinals,
$X$ is a topological space and $D$ is an ultrafilter. A cardinal $\mu $ is 
also considered as a topological space
endowed either with the order topology,
or with the coarser \emph{initial interval} topology  $\iit$,
the topology consisting of the intervals of the kind $[0, \alpha )$ with $\alpha \leq \mu$.
No separation axioms are assumed throughout.
Products  of topological spaces are always assigned the Tychonoff topology.
A topological  space $X$ is
\emph{$[\mu, \lambda ]$-compact} if  every open 
cover of $X$ by at most $\lambda$  sets has a subcover by less than $\mu$  sets.
Final $\kappa$-compactness is equivalent to 
$[ \nu , \nu ]$-compactness for every 
$\nu \geq \kappa $. More generally,
$[\mu, \lambda ]$-compactness is equivalent to 
$[ \nu , \nu ]$-compactness for every 
$\nu $ such that $\mu \leq \nu \leq \lambda $.
The \emph{infinitary language} $\mathcal L _{ \omega _1, \omega   } $ is 
like first-order logic, except that we allow
conjunctions 
and disjunctions 
of countably many formulas. 
If $\Sigma$  and $\Gamma$ are sets of sentences of 
$\mathcal L _{ \omega _1, \omega } $
we say that $\Gamma$ is \emph{$\mu$-satisfiable
relative to $\Sigma$} if
$\Sigma \cup \Gamma'$
is satisfiable, for every 
$\Gamma' \subseteq \Gamma $ 
of cardinality $<\mu$.
If $\mu \leq \lambda $ we say that $\mathcal L _{ \omega _1, \omega } $ is  
\emph{$ \kappa  $-$( \lambda   , \mu  )$-compact}
if $\Sigma \cup \Gamma$
is satisfiable, whenever 
$|\Sigma| \leq \kappa  $, $|  \Gamma| \leq \lambda $
and $\Gamma$ is $\mu$-satisfiable
relative to $\Sigma$.
We had formerly introduced the notion of $ \kappa  $-$( \lambda , \mu )$-compactness
for arbitrary logics,  extending  notions by
C. C. Chang, H. J. Keisler, J. A. Makowsky, S. Shelah, A. Tarski, among others.
See the book edited by  Barwise and Feferman \cite{BF}, Caicedo \cite{Ca} 
and our \cite{cpnw} 
for references.  
If $ \kappa \leq \lambda $
then $ \kappa  $-$( \lambda , \mu   )$-compactness reduces to
the classical notion of $( \lambda , \mu   )$-compactness.
Notice the reversed order of the cardinal parameters in comparison with  
the corresponding topological property.
If $D$ is an ultrafilter over some set $I$ and $f \co I \to J$ is a function, $f(D)$ is the ultrafilter over $J$ 
defined by $Y \in f(D)$ if and only if $ f ^{-1}(Y) \in D $.
As usual, $[\lambda] ^{< \lambda }$ denotes the set of 
all subsets of $\lambda$ of cardinality $< \lambda $. 
We say that an ultrafilter $D$ over
$[\lambda] ^{< \lambda }$
\emph{covers $\lambda$} if 
 $\{ s \in [\lambda ] ^{< \lambda } \mid \alpha \in s\} \in D$,
for every  $\alpha \in \lambda $. 

\begin{definition} \labbel{defb}    
We shall denote by
$(\lambda, \lambda ) {\not\Rightarrow} (\mu_ \gamma ) _{ \gamma \in \kappa } $ 
the following statement.
 \begin{enumerate}   
\item[(*)]
For every sequence of functions 
$(f_ \gamma  ) _{ \gamma \in \kappa }$
such that  $f_ \gamma   \co [\lambda] ^{< \lambda }  \to \mu_ \gamma  $
for $\gamma \in \kappa$,
there is some ultrafilter $D$ over $[\lambda] ^{< \lambda } $ 
covering $\lambda$ such that
for no $\gamma \in \kappa $
$f_ \gamma (D)$ is uniform over $\mu_ \gamma $.    
 \end{enumerate}

We shall write 
$(\lambda, \lambda ) {\not\Rightarrow}\!^ \kappa  \mu$ 
when all the $\mu_ \gamma $'s in (*) are equal to $\mu$.

The negation of  $(\lambda, \lambda ) {\not\Rightarrow}\!^ \kappa  \mu$  is denoted by $( \lambda , \lambda) {\Rightarrow}\!^ \kappa  \mu$; similarly for
 $(\lambda, \lambda ) {\Rightarrow} (\mu_ \gamma ) _{ \gamma \in \kappa } $. 
\end{definition}

If $D$ is an ultrafilter over some set $I$, 
a point $x \in X$ is said to be a \emph{$D$-limit point} of
a sequence 
$(x _ i) _{i \in I } $ of elements of $X$ 
if 
 $\{ i \in I \mid x_i \in U\} \in D$,
for every open neighborhood $U$ of $x$.
To avoid
complex expressions in subscripts, 
we sometimes shall denote
a sequence 
$(x_i) _{i \in I} $ 
as 
 $ \langle  x_i \mid  i \in I \rangle  $.
The next theorem follows easily
from Caicedo \cite[Section 3]{Ca},
which extended, generalized and simplified former results by  A. R. Bernstein, J. Ginsburg and V. Saks,
among others. 
A detailed proof in even more general contexts can be found 
in Lipparini \cite[Proposition 32 (1) $\Leftrightarrow $  (7)]{tapp2}, 
taking $\mathcal F$ there to be the set of all singletons of $X$,
and in 
\cite[Theorem 2.3]{ufmeng},
taking $\lambda=1$ there. 

\begin{theorem} \labbel{caicb}
A  topological space $X$ is
 $[ \lambda , \lambda ]$-\brfrt compact 
if and only if for every sequence 
$ \langle  x_s \mid  s \in [\lambda] ^{< \lambda } \rangle $ 
of elements of $X$
there exists some ultrafilter $D$ over 
$[\lambda] ^{< \lambda }$
such that $D$ covers $\lambda$ and the sequence
has some $D$-limit point in $X$.  
 \end{theorem}

\begin{cor} \labbel{produfb}
Suppose that
$(\mu_ \gamma ) _{ \gamma \in \kappa } $
 is a sequence of regular cardinals
and that each $\mu _ \gamma $ is endowed either with the
order topology or with the $\iit$ topology. 
Then 
$\prod _{ \gamma \in \kappa } \mu _ \gamma   $ is  $[ \lambda, \lambda ]$-compact
if and only if $( \lambda , \lambda) {\not\Rightarrow} (\mu_ \gamma ) _{ \gamma \in \kappa } $. 
 \end{cor}

\begin{proof}
Let $X=\prod _{ \gamma \in \kappa } \mu _ \gamma $ and, for 
$ \gamma \in \kappa$, let 
$\pi_ \gamma \co X \to \mu_ \gamma $ be the natural projection.
A sequence of functions as in the first line of (*) in Definition \ref{defb} 
can be naturally identified with a sequence
$ \langle  x_s \mid  s \in [\lambda] ^{< \lambda } \rangle $  of elements 
of $X $, by posing
$\pi _ \gamma (x_ s)=f_ \gamma ( s )  $. 
By Theorem \ref{caicb},
$X$
is $[ \lambda, \lambda ]$-compact if and only if, 
for every  sequence 
$ \langle  x_s \mid  s \in [\lambda] ^{< \lambda } \rangle $ 
of elements of $X$, there is an ultrafilter $D$  over $[\lambda] ^{< \lambda }$  covering $\lambda$ and 
such that 
$ \langle  x_s \mid  s \in [\lambda] ^{< \lambda } \rangle $ 
has a  $D$-limit point in $ X$.
Since a sequence in a product  has a $D$-limit point 
 if and only if each component has a $D$-limit point,
 the above condition holds if and only if,
 for each $\gamma \in \kappa $,
$ \langle  \pi _ \gamma (x_s) \mid  s \in [\lambda] ^{< \lambda } \rangle $ 
 has a $D$-limit point in $\mu_ \gamma $, and
this happens if and only if,   for each $\gamma \in \kappa $,
there is $\delta_ \gamma \in \mu _ \gamma $
such that 
$\{  s \in [\lambda] ^{< \lambda } \mid \pi _ \gamma (x_ s) < \delta _ \gamma \} \in D$---no matter whether $\mu _ \gamma $ has the order or the $\iit$ topology.
Under the  mentioned identification, and since 
every $\mu_ \gamma $ is regular, 
this
means exactly that  each  $f_ \gamma (D)$ fails to be  uniform over $\mu_ \gamma $.
 \end{proof}

\begin{definition} \labbel{a}   
We now need to  consider
a model $ \m A( \lambda, \mu )$
which contains both a copy of 
$ \langle [\lambda] ^{< \lambda }, {\subseteq,}  \{\alpha  \}  \rangle _{ \alpha \in \lambda } $ 
and a copy of
$\langle  \mu , < , \beta  \rangle _{ \beta \in \mu  } $,
where 
the $ \{\alpha  \} $'s 
and the
$\beta  $'s
are interpreted as constants.
Though probably the most elegant way to accomplish this is by means of 
a two-sorted model, we do not want to introduce technicalities and simply assume
that $A= [\lambda] ^{< \lambda } \cup \mu  $
and that $[\lambda] ^{< \lambda }$ and $\mu $ are interpreted,
respectively, by unary predicates  
$U$ and $V$. 
In details, we let
$ \m A( \lambda, \mu )  = 
\langle A, U, V  , \subseteq , <,  \{\alpha  \} , \beta  \rangle
 _{ \alpha \in \lambda, \beta \in \mu  } $,
where 
$U  (s)$ holds in $\m A( \lambda, \mu )$ if and only if $s \in [\lambda] ^{< \lambda }$ and
$V (c)$ holds in $\m A( \lambda, \mu )$ if and only if $ c \in \mu $.
By abuse of notation we shall not distinguish between symbols and
their interpretations.
If $\m A$ is an expansion of $\m A( \lambda, \mu )$ and $\m B \equiv \m A$ (that is, $\m B$ is \emph{elementarily equivalent}
to $\m A$), we say that $b \in B$  \emph{covers} $\lambda$  
if $U(b)$ and $ \{ \alpha  \} \subseteq  b$ hold in $\m B$, for every $\alpha \in \lambda $.
We say that    
$c \in B$ is \emph{$\mu$-nonstandard} 
if $  V(c) $ and $\beta < c $ hold  in $\m B$, for every $\beta \in \mu$.
Of course, in the case $\mu= \omega $,
we get the usual notion of a nonstandard element.
Notice that if $D$ is an ultrafilter over
$[\lambda] ^{< \lambda }$ then $D$ covers $\lambda$ in the ultrafilter sense
(cf. the sentence immediately before Definition \ref{defb}) 
if and only if the $D$-class $[Id]_D$ of the identity 
on 
$[\lambda] ^{< \lambda }$
in the ultrapower
$\prod _D \m A( \lambda, \mu )$
covers $\lambda$ in the present sense.
Moreover, if $\mu $ is regular,
then an ultrafilter $D$ over $\mu $ is uniform
if and only if  
$\prod _D \m A( \lambda, \mu )$ has a $\mu $-nonstandard element.
 \end{definition}

\begin{thm} \labbel{nonstb}
If 
$\kappa \geq \sup \{\lambda, \mu \}  $  then 
$(\lambda, \lambda ){\not\Rightarrow}\!^ \kappa  \mu$
if and only if for 
every expansion $\m A$ of 
$ \m A( \lambda, \mu )  $ 
with at most $\kappa$ new symbols
(equivalently, symbols and sorts), there is 
$\m B \equiv \m A$ 
such that $\m B$ has an element covering $\lambda$ 
but no $\mu$-nonstandard element.
 \end{thm}

 \begin{proof} 
Suppose that $(\lambda, \lambda ){\not\Rightarrow}\!^ \kappa  \mu$
and let $\m A$ be an expansion of 
$ \m A( \lambda, \mu )  $
with at most $\kappa$ new symbols and sorts.
Without loss of generality we may assume that 
$\m A$ has Skolem functions, 
since this adds at most $\kappa \geq \sup \{\lambda, \mu \}  $ new symbols.
Enumerate as $(f_ \gamma  ) _{ \gamma \in \kappa }$
all the functions
from  $[\lambda] ^{< \lambda }$ to $  \mu$
which are definable in $\m A$
(repeat occurrences, if necessary), 
and let $D$ be the ultrafilter given by 
 $(\lambda, \lambda ) {\not\Rightarrow}\!^ \kappa  \mu$.
Let $\m C$ be the ultrapower $\prod _D \m A$.
By the remark before the statement of the theorem,
$b = [Id]_D$ is an element in 
$\m C$ which covers $\lambda$.
 Let $\m B$ be the Skolem hull of 
$ \{ b \} $ in $\m C$; thus $\m B \equiv \m C = \prod _D \m A \equiv \m A$,
 and $b$ covers $\lambda$ in $\m B$. 
Had $\m B$ a $\mu$-nonstandard element $c$,
there  would be $\gamma \in \kappa $ such that $c= f_ \gamma (b)$,
by the definition of $\m B$.  Thus 
$c= f_ \gamma ([Id]_D)= [f_ \gamma ]_D$,
but this would imply
that $f_ \gamma (D)$ is uniform over $ \mu$,
since $\mu $ is regular,  
contradicting the choice of $D$.  

For the converse, suppose that $(f_ \gamma  ) _{ \gamma \in \kappa }$
is a sequence of functions from $ [\lambda] ^{< \lambda }$ to $\mu$.
Let $\m A$ be the expansion of 
$ \m A( \lambda, \mu )  $ 
 obtained  
by adding the $f_ \gamma $'s as unary functions.
Notice that we have no need to introduce new sorts.
By assumption, there is $ \m B \equiv  \m A$ 
 with an element $b$ covering $\lambda$ 
but without  $\mu$-nonstandard elements.
For every formula $\varphi(y)$ in the vocabulary of
$\m A$ and
with exactly one free variable $y$, 
let $Z_ \varphi = \{ s\in [\lambda] ^{< \lambda }
 \mid \varphi(s) \text{ holds in } \m A\}$.
Put 
$E=\{ Z_ \varphi \mid \varphi \text{ is as above and } 
\varphi(b) \text{ holds in  } \m B \}$. 
$E$ has trivially the finite intersection property, thus it can be extended to some ultrafilter 
$D$ over $[\lambda] ^{< \lambda }$. 
For each $\alpha \in \lambda $,  considering the formula
$ \varphi_ \alpha \co  \{ \alpha  \} \subseteq y $,
we get that 
$ Z _{ \varphi _ \alpha } = \{ s\in [\lambda] ^{< \lambda }
 \mid \alpha \in s \} \in E \subseteq D$, thus
$D$ covers $\lambda$.  
Let $\gamma \in \kappa $. Since  $\m B$ has no 
 $\mu$-nonstandard element,
there is $\beta < \mu$ such that 
$  f_ \gamma (b)< \beta  $  holds in  $  \m B$.  
Letting $\varphi _ \gamma ( y)$
be $  f_ \gamma (y)< \beta  $, we get that
 $Z _{  \varphi_ \gamma } = 
\{ s\in [\lambda] ^{< \lambda } \mid  f_ \gamma (s ) < \beta \}
\in E \subseteq D$, 
thus 
$[0, \beta ) \in f_ \gamma (D)$, 
proving 
that 
$  f_ \gamma (D)$ is not uniform over $\mu$.
\end{proof} 

Theorem \ref{nonstb} explains the reason why we have used the negation of an implication sign in the notation 
$(\lambda, \lambda ){\not\Rightarrow}\!^ \kappa  \mu$.
 The principle asserts that,
modulo possible expansions, the existence of an element 
covering $\lambda$ does \emph{not} necessarily \emph{imply}
the existence of a $\mu $-nonstandard element. 
Similarly, 
$(\lambda, \lambda ){\not\Rightarrow}\!^ \kappa  \mu$
is equivalent to the statement that 
$[ \lambda , \lambda  ]$-compactness of a product of
$\kappa$-many topological spaces does not imply
$[ \mu  , \mu   ]$-compactness of a factor.  
See Proposition \ref{tuttxb}.

\begin{proof}[Proof of Theorem \ref{logb}]
The first statement is immediate from the  case $\kappa= \lambda $
of (1) $\Leftrightarrow $  (2), since 
$ \omega^ \lambda $
is finally $\lambda^+$-compact, 
having a base of cardinality $\lambda$.
The equivalence of (1) and (3) is the particular case of 
Corollary  \ref{produfb} when all $\mu_ \gamma $'s equal $ \omega$. Thus, 
in view of Theorem \ref{nonstb}, it is enough to prove
that (2) is equivalent to the necessary and sufficient condition given there for 
$( \lambda ,\lambda) {\not\Rightarrow}\!^ \kappa \omega $.
For the simpler direction,
suppose that 
 $\mathcal L _{ \omega _1, \omega } $
is $ \kappa  $-$( \lambda , \lambda  )$-compact
 and that
 $\m A$ is an expansion of 
$ \m A( \lambda, \omega  )  $ 
with at most $\kappa$ new symbols.
Let $\Sigma$ be the elementary (first order) theory of 
$\m A$ plus an   $\mathcal L _{ \omega _1, \omega } $ sentence
asserting that there exists no nonstandard element
and let $\Gamma = \{ \{ \alpha \} \subseteq b \mid \alpha \in \lambda  \} $. 
By applying 
$ \kappa  $-$( \lambda , \lambda  )$-compactness of  
$\mathcal L _{ \omega _1, \omega } $ to the above sets of sentences
we get a model $\m B$ as requested by the condition in  
Theorem \ref{nonstb}. 

The reverse direction is a variation on a standard reduction argument.
Suppose that the condition
in Theorem \ref{nonstb} holds. 
If $\m A$ is a many-sorted expansion of the model
$ \m A( \lambda, \omega  )  $ introduced in  Definition \ref{a}
then,
for every $\m B \equiv \m A$
such that  $\m B$ 
has no  nonstandard element, 
a formula $\psi$ of $\mathcal L _{ \omega _1, \omega } $ of 
the form $\bigwedge _{n \in \omega } \varphi_n (\bar{x})  $ 
is equivalent  
to $\forall y (V(y) \Rightarrow R_ \psi(y, \bar{x}))$
in some expansion $\m B^+$ of $\m B$
with a newly introduced relation 
$R_ \psi$ such that 
$ \forall  \bar{x} (R_ \psi(n, \bar{x}) \Leftrightarrow \varphi_n (\bar{x}) )$
holds in  $\m B^+$, for every 
$ n \in \omega$. 
Here we are using in an essential way the fact that
in  a sentence of $\mathcal L _{ \omega _1, \omega } $  we can quantify away
only a finite number of variables, hence we can do 
with a finitary relation $R_ \psi$.
Thus,
given 
 $\Sigma  $ 
and $\Gamma $ sets of 
sentences as in the definition of $ \kappa  $-$( \lambda , \lambda  )$-compactness, 
 assuming that we work in models 
without nonstandard elements,
iterating the above procedure for all subformulas of the sentences
under consideration
and working in some appropriately expanded vocabulary, 
 we may reduce all the relevant satisfaction conditions to the case in which 
 $\Sigma  $ 
and $\Gamma $ are sets of first order sentences.
In the situation at hand we need to add to $\Sigma$ all the sentences of the 
form $ \forall \bar{x}(R_ \psi(n, \bar{x}) \Leftrightarrow \varphi_n (\bar{x})) $
as above, but easy computations show that we still have $|\Sigma| \leq \kappa  $, since $\kappa \geq \lambda $. 
If 
$\Gamma = \{ \gamma _ \alpha \mid \alpha \in \lambda  \} $ is $\lambda$-satisfiable relative to $\Sigma$, 
construct a many-sorted expansion 
$ \m A$
of  $ \m A( \lambda, \omega  )  $
  with a further new binary relation
$S$ such that, for every $ s \in [\lambda] ^{< \lambda } $,
  $\{ z \in A \mid S( s, z) \}$ 
models $\Sigma \cup \{ \gamma _ \alpha \mid \alpha \in s  \} $.
This is possible, since $\Gamma$ is $\lambda$-satisfiable
relative to $\Sigma$.
If  $\m B \equiv \m A$ is given by
$( \lambda , \lambda) {\not\Rightarrow}\!^ \kappa \omega $
 and $b \in B$ covers  $\lambda$ then  
$\{ z \in B \mid S( b , z) \}$ models $\Sigma \cup \Gamma$.
Indeed,
 for every $\alpha \in \lambda $ the sentence
$ \forall w( \{ \alpha \} \subseteq w 
\wedge U(w)
\Rightarrow 
\gamma _ \alpha ^{S(w, -)}) $ is satisfied in 
$\m A$ hence, 
by elementarity,
it is  satisfied in $\m B$, too,
where by 
$\gamma _ \alpha ^{S(w, -)} $
we denote a \emph{relativization} 
of $\gamma _ \alpha $ to $ S(w, -) $,
that is, a sentence such that  if $\m C$ is a model, $c \in C$  and
 $\{ z \in C \mid S( c , z) \}$ is itself a model 
for an appropriate vocabulary, then
$\gamma _ \alpha ^{S(c, -)} $ is satisfied in $\m C$ 
if and only if 
$\gamma _ \alpha $
is satisfied in
$\{ z \in C \mid S( c , z) \}$.
Similarly, for every  $\sigma \in \Sigma $,
 $ \forall w( 
 U(w)
\Rightarrow 
\sigma ^{S(w, -)}) $ is satisfied in 
$\m A$ hence 
it is  satisfied in $\m B$.
See the book edited by  Barwise and Feferman \cite{BF}
for further technical details, 
in particular about relativization and about 
dealing with constants.
Notice also that in the above proof we do need the many-sorted (or relativized)
version of the condition in Theorem \ref{nonstb},
since when $\kappa$ is large the models witnessing the $\lambda$-satisfiability  
of $\Gamma$ relative to $\Sigma$
might have cardinality exceeding the cardinality of the base set of
$ \m A( \lambda, \omega  )  $.
\end{proof}

In \cite{cpnw} we have proved results similar to those presented here, 
but we have dealt mostly with regular $\lambda$.
In that situation we could
do with a principle simpler than  
$(\lambda, \lambda ) {\not\Rightarrow} (\mu_ \gamma ) _{ \gamma \in \kappa } $;
we have denoted that principle by 
$ \lambda  {\not\Rightarrow} (\mu_ \gamma ) _{ \gamma \in \kappa } $.
The definition is essentially as Definition \ref{defb},
except for the fact that this time we use uniform ultrafilters over $\lambda$
rather than ultrafilters over $[\lambda] ^{< \lambda } $ covering $\lambda$. 
For $\lambda$ regular the two principles are easily seen to be equivalent.
Indeed, if $\lambda$ is regular,
$f_1 \co [\lambda] ^{< \lambda } \to \lambda $ 
is defined by $f_1(s)= \sup s$, and
an ultrafilter $D$ over $[\lambda] ^{< \lambda } $ covers $\lambda$ 
then $f_1(D)$ is uniform over $\lambda$.
The assumption that $\lambda$ is regular is needed to get that the range of
$f_1$ is contained in $\lambda$. 
On the other hand, the function  $f_2 \co   \lambda\to  [\lambda] ^{< \lambda } $  
which assigns to $\alpha \in \lambda $ the set
$\{ \beta \in \lambda \mid \beta < \alpha  \}$ is such that
if an ultrafilter $D$ over $\lambda$ is uniform
then $f_2(D)$  over $[\lambda] ^{< \lambda } $
covers $\lambda$.
Of course, under 
the standard nowadays usual identification of an ordinal
with the set of all smaller ordinals, 
 $f_2 ( \alpha ) = \alpha $.
 Using $f_1$ and  $f_2$
as tools for trasferring  
 from $[\lambda] ^{< \lambda } $ to $\lambda$ and conversely,
one immediately  sees that
 $ \lambda  {\not\Rightarrow} (\mu_ \gamma ) _{ \gamma \in \kappa } $
implies
$(\lambda, \lambda ) {\not\Rightarrow} (\mu_ \gamma ) _{ \gamma \in \kappa } $
for every $\lambda$, and that the two principles
 are equivalent for $\lambda$ regular.

The above principles
are interesting only for small values of $\kappa$.
Indeed if $\kappa \geq \mu ^{ \lambda ^{< \lambda } } $
then  $( \lambda ,\lambda) {\not\Rightarrow}\!^ \kappa \mu$ is equivalent to the statement that there exists an ultrafilter  over $[\lambda] ^{< \lambda } $ 
 covering $\lambda$ such that for no function 
$f \co [\lambda] ^{< \lambda } \to \mu  $
$f(D)$ is uniform over $\mu $. This is because
there are exactly $\mu ^{ \lambda ^{< \lambda } } $
functions from 
$[\lambda] ^{< \lambda } $ to $\mu $.
Thus while in general we obtain a stronger
statement when we increase $\kappa$ 
in $( \lambda ,\lambda) {\not\Rightarrow}\!^ \kappa \mu$,
at the  point $\kappa = \mu ^{ \lambda ^{< \lambda } } $
we have already reached the strongest possible notion.
Recall that an ultrafilter $D$ over a set $I$ is \emph{$( \lambda , \lambda ) $-regular} if 
there is a function $f \co I \to [\lambda] ^{< \lambda }  $ such that 
$f(D)$ covers $\lambda$; and that $D$ is  \emph{$ \mu $-decomposable} if 
there is a function $f \co I \to \mu  $ such that 
$f(D)$ is uniform over $ \mu $.  
Hence for $\kappa \geq \mu ^{ \lambda ^{< \lambda } } $
we get that $( \lambda ,\lambda) {\not\Rightarrow}\!^ \kappa \mu$
is equivalent to the statement that there is a 
$( \lambda , \lambda ) $-regular not
 $\mu $-decomposable ultrafilter.
The problem of the existence of such ultrafilters
is connected with difficult set-theoretical problems involving
large cardinals, forcing and pcf-theory, and has been widely studied,
sometimes in equivalent formulations. See 
\cite{mru} for more information.
In a couple of papers we have 
somewhat attempted a study of
the more comprehensive (hence more difficult!)  relation $( \lambda ,\lambda) {\not\Rightarrow}\!^ \kappa \mu$.
References can be found in \cite{mru,cpnw}.
Roughly, while for large $\kappa$ we get notions related to measurability,
on the other hand, for smaller values of $\kappa$ we get corresponding variants of weak compactness,
as the present note itself exemplifies.
Notice that in some previous works 
we had given the definition of, say, 
 $(\lambda, \lambda ) {\not\Rightarrow}\!^ \kappa \mu$
by means of the equivalent condition
given here by Theorem \ref{nonstb},
in which the assumption
$ \kappa \geq \sup \{ \lambda, \mu  \} $ is made. 
Hence in some places  the notation we have used
 might be not consistent with the present one
(but only when small values of $\kappa$
are taken into account).

Finally, expanding a remark we have  presented in \cite{cpnw}, we notice that, though we have stated our topological results in terms of products of cardinals,
 they can be reformulated in a way that  involves  arbitrary
products of topological spaces.

\begin{proposition} \labbel{tuttxb}
If
$(\mu_ \gamma ) _{ \gamma \in \kappa } $
 is a sequence of regular cardinals then the following conditions are equivalent. \begin{enumerate}   
\item 
$\prod _{ \gamma \in \kappa } \mu _ \gamma   $ is not $[ \lambda, \lambda ]$-compact, where each $\mu _ \gamma $ is equivalently endowed either with the
order topology or with the $\iit$ topology. 
\item
$( \lambda , \lambda) {\Rightarrow} (\mu_ \gamma ) _{ \gamma \in \kappa } $
\item  
For every set $I$ and every product $X= \prod _{i \in I} X_i$ 
of topological spaces, if $X$ is 
$ [\lambda , \lambda  ]$-compact, then
for every injective function $g \co  \kappa \to I$ 
there is $\gamma \in \kappa $ such that 
$X _{g( \gamma )} $ is  $[\mu _ \gamma , \mu _ \gamma  ]$-compact.
 \end{enumerate} 
 \end{proposition}

\begin{proof}
(1) $\Leftrightarrow $   (2) is  Corollary \ref{produfb} in contrapositive form.
(3) $\Rightarrow $  (1) is
trivial by taking $I= \kappa $ and $g$ to be the identity function
 and observing that $\mu _ \gamma$ is not $[\mu _ \gamma , \mu _ \gamma  ]$-compact (with either topology), since $\mu _ \gamma$ is regular. 
To finish the proof we shall prove that if (3) fails then (1) fails.
So suppose that (3) fails as witnessed by some $ [\lambda , \lambda  ]$-compact $X= \prod _{i \in I} X_i$
and an injective $g \co  \kappa \to I$ such that no $X _{g( \gamma )} $ is  $[\mu _ \gamma , \mu _ \gamma  ]$-compact.
It is easy to see that if $\mu $ is regular then a space
is not $[\mu , \mu ]$-compact if and only if there is a continuous surjective
function $h \co X \to \mu $ with the $\iit$ topology; see, e.~g., Lipparini \cite[Lemma 4]{1}. 
Hence for each $\gamma \in \kappa $  we have 
a continuous surjective
function $h _ \gamma  \co X_ \gamma  \to \mu _ \gamma  $
and, by naturality of products and since $g$ is injective, a continuous surjective
$h   \co X  \to  \prod _{ \gamma \in \kappa } \mu _ \gamma   $.
Thus if $X$ is $ [\lambda , \lambda  ]$-compact then so is
$\prod _{ \gamma \in \kappa } \mu _ \gamma $, since 
$ [\lambda , \lambda  ]$-compactness is preserved 
under surjective continuous images, hence (1) fails in the case 
the $\mu _ \gamma $'s are assigned the $\iit$ topology. This is enough,
since we have already proved that in (1) we can equivalently consider either topology, 
since in each case (1) is equivalent to (2). 
 \end{proof} 

Notice that, in particular, it follows from 
Proposition \ref{tuttxb} that if
 $\mu $ is regular, $(\lambda, \lambda ) {\Rightarrow}\!^ \kappa  \mu$
and some product is $ [\lambda , \lambda  ]$-compact, 
then all but at most $< \kappa$ 
factors are $ [\mu , \mu  ]$-compact.
In this way we obtain alternative proofs--- 
as well as various strengthenings---of many of the results
we have proved  in \cite{tapp}.

\end{document}